\documentclass{article}
\usepackage{amsfonts}
\usepackage{amsmath}

\newtheorem{theorem}{Theorem}

\newtheorem{definition}[theorem]{Definition}

\newtheorem{lemma}[theorem]{Lemma}

\newtheorem{remark}[theorem]{Remark}

\newenvironment{proof}[1][Proof]{\noindent\textbf{#1.} }{\ \rule{0.5em}{0.5em}}
\input{tcilatex}
\begin{document}

\begin{center}
\bigskip

{\Large Approximation properties for Baskakov-Kantorovich-Stancu type
operators based on }${\Large q-}${\Large \ integers\bigskip }

\c{C}i\u{g}dem Atakut$^{1}$, \.{I}brahim B\"{u}y\"{u}kyaz\i c\i $^{2}$
\bigskip

$^{1}${\small Ankara University, Faculty of Science, Department of
Mathematics, Tandogan 06100, Ankara, Turkey. }

{\small e-mail: atakut@science.ankara.edu.tr\bigskip }

$^{2}${\small Gazi University, Department of Mathematics, \c{C}ubuk, Ankara,
Turkey.}

{\small \ e-mail: bibrahim@gazi.edu.tr}
\end{center}

\textit{Abstract: }In this paper, we give an interesting generalization of
the Stancu type Baskakov-Kantorovich operators based on the $q-$integers and
investigate their approximation properties. Also, we obtain the estimates
for the rate of convergence for a sequence of them by the weighted modulus
of smoothness.\bigskip

\bigskip 


\section{\protect\bigskip Introduction}

In recent years, due to the intensive development of $q-$ calculus,
generalizations of some operators related to $q-$ calculus have emerged (see 
$[2,5,6,10-18]$). Aral and Gupta defined $q-$ generalization of the Baskakov
operator and investigated approximation properties of these operators in $%
[2] $. In $[12],$ Gupta and Radu introduced the Baskakov- Kantorovich
operators based on $q-$integers and investigated their weighted statistical
approximation properties. They also proved some direct estimations for error
using weighted modulus of smoothness in case $0<q<1.$ In recent study B\"{u}y%
\"{u}kyaz\i c\i\ and Atakut $[5]$ introduced a new Stancu type
generalization of $q-$ Baskakov operator is defined as

\begin{equation}
L_{n}^{\alpha ,\beta }(f;q,x)=\dsum\limits_{k=0}^{\infty }q^{\frac{k(k-1)}{2}%
}\dfrac{D_{q}^{k}\left( \varphi _{n}(x)\right) }{\left[ k\right] _{q}!}%
(-x)^{k}\ \ f(\frac{1}{q^{k-1}}\frac{\left[ k\right] _{q}+q^{k-1}\alpha }{%
\left[ n\right] _{q}+\beta })\   \tag{1}
\end{equation}%
where\ $0\leq \alpha \leq \beta ,$ $q\in (0,1),\ f\in C[0,\infty )$ and the
following conditions are provided:

Let $\left\{ \varphi _{n}\right\} $ $(n=1,2,...)$ $\varphi _{n}:%
%TCIMACRO{\U{211d} }%
%BeginExpansion
\mathbb{R}
%EndExpansion
\rightarrow 
%TCIMACRO{\U{211d} }%
%BeginExpansion
\mathbb{R}
%EndExpansion
$ \ be a sequence which is satisfying following conditions,

$(i)$ $\varphi _{n}$ $(n=1,2,...),$ $k-$ times continuously $q-$
differentiable any closed interval $\left[ 0,b\right] $,

$(ii)$ $\varphi _{n}(0)=1,(n=1,2,...),$

$(iii)$ for all $x\in \left[ 0,b\right] ,$ and $(k=0,1,...;n=1,2,...)$, $%
(-1)^{k}D_{q}^{k}\left( \varphi _{n}(x)\right) \geq 0,$

$(iv)$ there exists a positive integer $m(n)$, such that 
\begin{equation}
D_{q}^{k}\left( \varphi _{n}(x)\right) =-\left[ n\right] _{q}D_{q}^{k-1}%
\varphi _{m(n)}(x)(1+\alpha _{k,n}(x)),  \tag{2}
\end{equation}
$(k=0,1,...;n=1,2,...)$ and $x\in \left[ 0,b\right] $ where $\alpha
_{k,n}(x) $ converges to zero for $n\rightarrow \infty $ uniformly in $k.$

$(v)\lim\limits_{n\rightarrow \infty }\frac{\left[ n\right] _{q}}{\left[ m(n)%
\right] _{q}}=1$.

Now, to explain the construction of the new $q-$ operators, we mention some
basic definitions of $q-$ calculus and Lemma.

Let $q>0$. For each nonnegative integer $n$, we define the $q-$ integer $%
[n]_{q}$ as 
\begin{equation*}
\lbrack n]_{q}=%
\begin{cases}
(1-q^{n})/(1-q) & \mbox{if}\quad q\neq 1 \\ 
n & \mbox{if}\quad q=1%
\end{cases}%
\end{equation*}%
and the $q-$ factorial $[n]_{q}!$ as 
\begin{equation*}
\lbrack n]_{q}!=%
\begin{cases}
\lbrack n]_{q}[n-1]_{q}\cdots \lbrack 1]_{q} & \mbox{if}\quad n\geq 1 \\ 
1 & \mbox{if}\quad n=0%
\end{cases}%
\end{equation*}%
For the integers $n$ and $k$, with $0\leq k\leq n$, the $q-$ binomial
coefficients are then defined as follows (see $[14]$): 
\begin{equation*}
\begin{bmatrix}
n \\ 
k%
\end{bmatrix}%
_{q}=\frac{[n]_{q}!}{[k]_{q}![n-k]_{q}!}.
\end{equation*}

Note that the following relation is satisfied

\begin{equation*}
\lbrack n]_{q}=[n-1]_{q}+q^{n-1}.
\end{equation*}

\begin{definition}
The $q-$ derivative of a function $f$ with respect to $x$ is 
\begin{equation*}
D_{q}\left( f(x)\right) =\dfrac{f(qx)-f(x)}{qx-x}
\end{equation*}%
which is also known as the Jackson derivative. High $q-$ derivatives are%
\begin{equation*}
D_{q}^{0}\left( f(x)\right) =f(x)\ ,\ D_{q}^{n}\left( f(x)\right)
=D_{q}\left( D_{q}^{n-1}\left( f(x)\right) \right) \ ,\ n=1,2,3,...
\end{equation*}
\end{definition}

Note that as $q\rightarrow 1$, the $q-$ derivative approach the usual
derivative.

\begin{definition}
The $q-$integration is defined as%
\begin{equation*}
\dint\limits_{0}^{a}f(t)d_{q}t=(1-q)a\dsum\limits_{j=0}^{\infty
}f(q^{j}a)q^{j}\ \ \ \ \ ,\ \ a>0.
\end{equation*}%
Over a general interval $\left[ a,b\right] $, $0<a<b$, one defines%
\begin{equation*}
\dint\limits_{a}^{b}f(t)d_{q}t=\dint\limits_{0}^{b}f(t)d_{q}t-\dint%
\limits_{0}^{a}f(t)d_{q}t.
\end{equation*}
\end{definition}

\begin{definition}
Let \ $f(x)$ be a continuous function on some interval $\left[ a,b\right] $
and $c\in \left( a,b\right) $. Jackson's $\ q-$ Taylor formula (see $\mathit{%
[}13,14])$ is given by 
\begin{equation*}
f(x)=\dsum\limits_{k=0}^{\infty }\dfrac{(D_{q}^{k}f)(c)}{\left[ k\right]
_{q}!}(x-c)_{q}^{k}
\end{equation*}%
where $(x-c)_{q}^{k}=\dprod\limits_{i=0}^{k-1}(x-cq^{i}).$
\end{definition}

First we need the following auxiliary result. Throughout the paper, we use $%
e_{i}$ the test functions defined by$~e_{i}(t):=t^{i}\ $for every integer$\
i\geq 0.$

\begin{lemma}[{from [5]}]
For \ $L_{n}^{\alpha ,\beta }(e_{i}(t);q,x)$, $i=0,1,2$ the following
identities hold:%
\begin{equation}
L_{n}^{\alpha ,\beta }(e_{0};q,x)=1,  \tag{3}
\end{equation}%
\begin{equation}
L_{n}^{\alpha ,\beta }(e_{1};q,x)=\frac{\left[ n\right] _{q}}{\left[ n\right]
_{q}+\beta }x(1+\alpha _{1,n}(x))+\frac{\alpha }{\left[ n\right] _{q}+\beta }%
,  \tag{4}
\end{equation}%
\begin{eqnarray}
L_{n}^{\alpha ,\beta }(e_{2};q,x) &=&\frac{\left[ n\right] _{q}\left[ m(n)%
\right] _{q}}{q\left( \left[ n\right] _{q}+\beta \right) ^{2}}x^{2}(1+\alpha
_{1,m(n)}(x))(1+\alpha _{2,n}(x))  \TCItag{5} \\
&&+\frac{\left[ n\right] _{q}\left( 2\alpha +1\right) }{\left( \left[ n%
\right] _{q}+\beta \right) ^{2}}x\left( 1+\alpha _{1,n}(x)\right) +\frac{%
\alpha ^{2}}{\left( \left[ n\right] _{q}+\beta \right) ^{2}}.  \notag
\end{eqnarray}
\end{lemma}

\section{Some properties of\ Stancu type \textit{q}-Baskakov-Kantorovich
operators}

In addition to the above conditions $(i)-(v),$ $\varphi _{n}(x)$ and $\alpha
_{k,n}(x)$ are satisfied following condition:%
\begin{equation*}
(vi)\text{ \ \ }\varphi _{n}(x)(1+\alpha _{0,n}(x))\leq 1~,~\ \text{for all }%
x\in \left[ 0,b\right] ,(n=1,2,...).
\end{equation*}

In this paper, under the conditions $(i)-(vi)$, we definition a new
generalization of Stancu type $q-$Baskakov-Kantorovich operators as following

\begin{equation}
L_{n}^{\ast (\alpha ,\beta )}(f;q,x)=\left( \left[ n\right] _{q}+\beta
\right) \dsum\limits_{k=0}^{\infty }q^{\frac{k(k-1)}{2}}\frac{%
D_{q}^{k}\left( \varphi _{n}(x)\right) }{\left[ k\right] _{q}!}\left(
-x\right) ^{k}\dint\limits_{q\left( \frac{\left[ k\right] _{q}+q^{k-1}\alpha 
}{\left[ n\right] _{q}+\beta }\right) }^{\frac{\left[ k+1\right]
_{q}+q^{k}\alpha }{\left[ n\right] _{q}+\beta }}f\left( q^{-k+1}t\right)
d_{q}t\ \ \ \ ,  \tag{6}
\end{equation}%
where $\ x\in 
%TCIMACRO{\U{211d} }%
%BeginExpansion
\mathbb{R}
%EndExpansion
_{+}\ ,n\in N,\ \ 0\leq \alpha \leq \beta .$

Note that, when $q=1$, the operators given by (6) is reduced to the
Kantorovich-Baskakov-Stancu type operators (see [3]) and if we choose $q=1$, 
$\varphi _{n}(x)=(1+x)^{-n}$ and $\alpha =\beta =0,$ we obtain
Baskakov-Kantorovich operators (see [1]).

In each of the following theorems, we assume that $q=$ $q_{n}$, where $%
\{q_{n}\}$ is a sequence of real numbers such that $0<q_{n}<1$ for all $n$
and $\lim_{n\rightarrow \infty }q_{n}=1.$

Now we give the following Lemmas, which are necessary to prove our theorems:

\begin{lemma}
The following relations are satisfied:
\end{lemma}

\begin{equation}
\dint\limits_{q\left( \frac{\left[ k\right] _{q}+q^{k-1}\alpha }{\left[ n%
\right] _{q}+\beta }\right) }^{\frac{\left[ k+1\right] _{q}+q^{k}\alpha }{%
\left[ n\right] _{q}+\beta }}d_{q}t=\frac{1}{\left[ n\right] _{q}+\beta },\ 
\tag{7}
\end{equation}

\begin{equation}
\dint\limits_{q\left( \frac{\left[ k\right] _{q}+q^{k-1}\alpha }{\left[ n%
\right] _{q}+\beta }\right) }^{\frac{\left[ k+1\right] _{q}+q^{k}\alpha }{%
\left[ n\right] _{q}+\beta }}td_{q}t=\frac{\left[ 2\right] _{q}\left[ k%
\right] _{q}+q^{k}(1+2\alpha )}{\left[ 2\right] _{q}\left( \left[ n\right]
_{q}+\beta \right) ^{2}}\ ,  \tag{8}
\end{equation}

\begin{equation}
\dint\limits_{q\left( \frac{\left[ k\right] _{q}+q^{k-1}\alpha }{\left[ n%
\right] _{q}+\beta }\right) }^{\frac{\left[ k+1\right] _{q}+q^{k}\alpha }{%
\left[ n\right] _{q}+\beta }}t^{2}d_{q}t=\frac{\left[ 3\right] _{q}\left[ k%
\right] _{q}^{2}+q^{k}\left[ k\right] _{q}\left( (1+3\alpha )\left[ 2\right]
_{q}+1\right) +\left( 1+3\alpha +3\alpha ^{2}\right) q^{2k}}{\left[ 3\right]
_{q}\left( \left[ n\right] _{q}+\beta \right) ^{3}}.  \tag{9}
\end{equation}

\begin{proof}
From properties of $q-$analogue integration, by simple computation we obtain 
$(7-9).$
\end{proof}

By the following Lemma Korovkin's conditions are satisfied.

\begin{lemma}
For all $x\in 
%TCIMACRO{\U{211d} }%
%BeginExpansion
\mathbb{R}
%EndExpansion
_{+}\ ,n\in N,\ \ \alpha ,\beta \geq 0$ and $0<q<1$ , we have%
\begin{equation}
L_{n}^{\ast (\alpha ,\beta )}(e_{0};q,x)=1,  \tag{10}
\end{equation}
\end{lemma}

\begin{equation}
L_{n}^{\ast (\alpha ,\beta )}(e_{1};q,x)=\frac{\left[ n\right] _{q}}{\left[ n%
\right] _{q}+\beta }x\left( 1+\alpha _{1,n}(x)\right) +\frac{q(1+2\alpha )}{%
\left[ 2\right] _{q}\left( \left[ n\right] _{q}+\beta \right) }\ ,  \tag{11}
\end{equation}

\begin{eqnarray}
L_{n}^{\ast (\alpha ,\beta )}(e_{2};q,x) &=&\frac{\left[ n\right] _{q}\left[
m(n)\right] _{q}}{q\left( \left[ n\right] _{q}+\beta \right) ^{2}}\left(
1+\alpha _{1,m(n)}(x)\right) \left( 1+\alpha _{2,n}(x)\right) x^{2} 
\TCItag{12} \\
&&+\frac{\left[ n\right] _{q}\left[ \left[ 3\right] _{q}+q\left( (1+3\alpha )%
\left[ 2\right] _{q}+1\right) \right] }{\left[ 3\right] _{q}\left( \left[ n%
\right] _{q}+\beta \right) ^{2}}\left( 1+\alpha _{1,n}(x)\right) x+\frac{%
q^{2}\left( 1+3\alpha +3\alpha ^{2}\right) }{\left[ 3\right] _{q}\left( %
\left[ n\right] _{q}+\beta \right) ^{2}}.  \notag
\end{eqnarray}

\begin{proof}
From definition $(6)$ and the identities $(3)$ and $(7)$, we can easily
obtain%
\begin{eqnarray*}
L_{n}^{\ast (\alpha ,\beta )}(e_{0};q,x) &=&\left( \left[ n\right]
_{q}+\beta \right) \dsum\limits_{k=0}^{\infty }q^{\frac{k(k-1)}{2}}\frac{%
D_{q}^{k}\left( \varphi _{n}(x)\right) }{\left[ k\right] _{q}!}\left(
-x\right) ^{k}\dint\limits_{q\left( \frac{\left[ k\right] _{q}+q^{k-1}\alpha 
}{\left[ n\right] _{q}+\beta }\right) }^{\frac{\left[ k+1\right]
_{q}+q^{k}\alpha }{\left[ n\right] _{q}+\beta }}d_{q}t \\
&=&L_{n}^{\alpha ,\beta }(e_{0};q,x)=1.
\end{eqnarray*}%
Now for $e_{1},$from (3), (4) and (8) we can write%
\begin{eqnarray*}
L_{n}^{\ast (\alpha ,\beta )}(e_{1};q,x) &=&\left( \left[ n\right]
_{q}+\beta \right) \dsum\limits_{k=0}^{\infty }q^{\frac{k(k-1)}{2}}\frac{%
D_{q}^{k}\left( \varphi _{n}(x)\right) }{\left[ k\right] _{q}!}\left(
-x\right) ^{k}\dint\limits_{q\left( \frac{\left[ k\right] _{q}+q^{k-1}\alpha 
}{\left[ n\right] _{q}+\beta }\right) }^{\frac{\left[ k+1\right]
_{q}+q^{k}\alpha }{\left[ n\right] _{q}+\beta }}q^{-k+1}td_{q}t \\
&=&\dsum\limits_{k=0}^{\infty }q^{\frac{k(k-1)}{2}}\frac{D_{q}^{k}\left(
\varphi _{n}(x)\right) }{\left[ k\right] _{q}!}\dsum\limits_{k=0}^{\infty
}q^{\frac{k(k-1)}{2}}\frac{\left[ k\right] _{q}+q^{k-1}\alpha }{%
q^{k-1}\left( \left[ n\right] _{q}+\beta \right) } \\
&&-\dsum\limits_{k=0}^{\infty }q^{\frac{k(k-1)}{2}}\frac{D_{q}^{k}\left(
\varphi _{n}(x)\right) }{\left[ k\right] _{q}!}\frac{q^{k-1}\alpha }{%
q^{k-1}\left( \left[ n\right] _{q}+\beta \right) } \\
&&+\frac{q(1+2\alpha )}{\left[ 2\right] _{q}\left( \left[ n\right]
_{q}+\beta \right) }\dsum\limits_{k=0}^{\infty }q^{\frac{k(k-1)}{2}}\frac{%
D_{q}^{k}\left( \varphi _{n}(x)\right) }{\left[ k\right] _{q}!}\left(
-x\right) ^{k} \\
&=&L_{n}^{\alpha ,\beta }(e_{1};q,x)-\frac{\alpha }{\left[ n\right]
_{q}+\beta }L_{n}^{\alpha ,\beta }(e_{0};q,x)+\frac{q(1+2\alpha )}{\left[ 2%
\right] _{q}\left( \left[ n\right] _{q}+\beta \right) }L_{n}^{\alpha ,\beta
}(e_{0};q,x) \\
&=&\frac{\left[ n\right] _{q}}{\left[ n\right] _{q}+\beta }x\left( 1+\alpha
_{1,n}(x)\right) +\frac{q(1+2\alpha )}{\left[ 2\right] _{q}\left( \left[ n%
\right] _{q}+\beta \right) }.
\end{eqnarray*}%
The finally, for $e_{2},$ we use (3), (4), (5) and (9), one has%
\begin{equation*}
L_{n}^{\ast (\alpha ,\beta )}(e_{2};q,x)=\left( \left[ n\right] _{q}+\beta
\right) \dsum\limits_{k=0}^{\infty }q^{\frac{k(k-1)}{2}}\frac{%
D_{q}^{k}\left( \varphi _{n}(x)\right) }{\left[ k\right] _{q}!}\left(
-x\right) ^{k}\dint\limits_{q\left( \frac{\left[ k\right] _{q}+q^{k-1}\alpha 
}{\left[ n\right] _{q}+\beta }\right) }^{\frac{\left[ k+1\right]
_{q}+q^{k}\alpha }{\left[ n\right] _{q}+\beta }}q^{-2k+2}t^{2}d_{q}t
\end{equation*}%
\begin{eqnarray*}
&=&\dsum\limits_{k=0}^{\infty }q^{\frac{k(k-1)}{2}}\frac{D_{q}^{k}\left(
\varphi _{n}(x)\right) }{\left[ k\right] _{q}!}\left( -x\right) ^{k}\left( 
\frac{\left[ k\right] _{q}+q^{k-1}\alpha }{q^{k-1}\left( \left[ n\right]
_{q}+\beta \right) }\right) ^{2} \\
&&-2\dsum\limits_{k=0}^{\infty }q^{\frac{k(k-1)}{2}}\frac{D_{q}^{k}\left(
\varphi _{n}(x)\right) }{\left[ k\right] _{q}!}\left( -x\right)
^{k}q^{-2k+2}q^{k-1}\alpha \frac{\left[ k\right] _{q}+q^{k-1}\alpha }{\left( %
\left[ n\right] _{q}+\beta \right) ^{2}} \\
&&+\dsum\limits_{k=0}^{\infty }q^{\frac{k(k-1)}{2}}\frac{D_{q}^{k}\left(
\varphi _{n}(x)\right) }{\left[ k\right] _{q}!}\left( -x\right) ^{k}q^{-2k+2}%
\frac{q^{2k-2}\alpha ^{2}}{\left( \left[ n\right] _{q}+\beta \right) ^{2}} \\
&&+\dsum\limits_{k=0}^{\infty }q^{\frac{k(k-1)}{2}}\frac{D_{q}^{k}\left(
\varphi _{n}(x)\right) }{\left[ k\right] _{q}!}\left( -x\right)
^{k}q^{-2k+2}\left( (1+3\alpha )\left[ 2\right] _{q}+1\right) \frac{\left[ k%
\right] _{q}+q^{k-1}\alpha }{\left[ 3\right] _{q}\left( \left[ n\right]
_{q}+\beta \right) ^{2}} \\
&&-\dsum\limits_{k=0}^{\infty }q^{\frac{k(k-1)}{2}}\frac{D_{q}^{k}\left(
\varphi _{n}(x)\right) }{\left[ k\right] _{q}!}\left( -x\right)
^{k}q^{-2k+2}q^{k}q^{k-1}\alpha \frac{\left( (1+3\alpha )\left[ 2\right]
_{q}+1\right) }{\left[ 3\right] _{q}\left( \left[ n\right] _{q}+\beta
\right) ^{2}} \\
&&+\dsum\limits_{k=0}^{\infty }q^{\frac{k(k-1)}{2}}\frac{D_{q}^{k}\left(
\varphi _{n}(x)\right) }{\left[ k\right] _{q}!}\left( -x\right) ^{k}\frac{%
q^{2}\left( 1+3\alpha +3\alpha ^{2}\right) }{\left[ 3\right] _{q}\left( %
\left[ n\right] _{q}+\beta \right) ^{2}}
\end{eqnarray*}%
\begin{eqnarray*}
&=&L_{n}^{\alpha ,\beta }(e_{2};q,x)-\frac{2\alpha }{\left[ n\right]
_{q}+\beta }L_{n}^{\alpha ,\beta }(e_{1};q,x)+\frac{\alpha ^{2}}{\left( %
\left[ n\right] _{q}+\beta \right) ^{2}}L_{n}^{\alpha ,\beta }(e_{0};q,x) \\
&&+\frac{q\left( (1+3\alpha )\left[ 2\right] _{q}+1\right) }{\left[ 3\right]
_{q}\left( \left[ n\right] _{q}+\beta \right) }L_{n}^{\alpha ,\beta
}(e_{1};q,x)-q\alpha \frac{\left( (1+3\alpha )\left[ 2\right] _{q}+1\right) 
}{\left[ 3\right] _{q}\left( \left[ n\right] _{q}+\beta \right) ^{2}}%
L_{n}^{\alpha ,\beta }(e_{0};q,x) \\
&&+\frac{q^{2}\left( 1+3\alpha +3\alpha ^{2}\right) }{\left[ 3\right]
_{q}\left( \left[ n\right] _{q}+\beta \right) ^{2}}L_{n}^{(\alpha ,\beta
)}(e_{0};q,x) \\
&=&\frac{\left[ n\right] _{q}\left[ m(n)\right] _{q}}{q\left( \left[ n\right]
_{q}+\beta \right) ^{2}}\left( 1+\alpha _{1,m(n)}(x)\right) \left( 1+\alpha
_{2,n}(x)\right) x^{2} \\
&&+\frac{\left[ n\right] _{q}\left[ \left[ 3\right] _{q}+q\left( (1+3\alpha )%
\left[ 2\right] _{q}+1\right) \right] }{\left[ 3\right] _{q}\left( \left[ n%
\right] _{q}+\beta \right) ^{2}}\left( 1+\alpha _{1,n}(x)\right) x+\frac{%
q^{2}\left( 1+3\alpha +3\alpha ^{2}\right) }{\left[ 3\right] _{q}\left( %
\left[ n\right] _{q}+\beta \right) ^{2}}.
\end{eqnarray*}%
This completes the proof of Lemma 6.
\end{proof}

Using above Lemma, we can obtain following theorem.

\begin{theorem}
Let $f\in C[0,b]$ , then 
\begin{equation*}
\lim_{n\rightarrow \infty }L_{n}^{\ast (\alpha ,\beta )}(f;q,x)=f(x)
\end{equation*}%
uniformly on $\left[ 0,b\right] .$
\end{theorem}

\section{Rate\ of\ convergence}

$B_{\rho _{\gamma }}(%
%TCIMACRO{\U{211d} }%
%BeginExpansion
\mathbb{R}
%EndExpansion
_{+}),$ the weighted space of real valued functions $f$ defined on $%
%TCIMACRO{\U{211d} }%
%BeginExpansion
\mathbb{R}
%EndExpansion
_{+}$ with the property $\left\vert f\left( x\right) \right\vert \leq
M_{f}\rho _{_{\gamma }}\left( x\right) $ where $\rho _{_{\gamma }}\left(
x\right) =1+x^{\gamma +2}$ and $M_{f}$ is constant depending on the function 
$f.$ We also consider the weighted subspace $C_{\rho _{_{\gamma }}}(%
%TCIMACRO{\U{211d} }%
%BeginExpansion
\mathbb{R}
%EndExpansion
_{+})$ of $B_{\rho _{_{\gamma }}}(%
%TCIMACRO{\U{211d} }%
%BeginExpansion
\mathbb{R}
%EndExpansion
_{+})$ given by 
\begin{equation*}
C_{\rho _{\gamma }}(%
%TCIMACRO{\U{211d} }%
%BeginExpansion
\mathbb{R}
%EndExpansion
_{+}):=\left\{ f\in B_{\rho _{_{\gamma }}}(%
%TCIMACRO{\U{211d} }%
%BeginExpansion
\mathbb{R}
%EndExpansion
_{+}):f~\ \text{continuous on }%
%TCIMACRO{\U{211d} }%
%BeginExpansion
\mathbb{R}
%EndExpansion
_{+}\right\} .
\end{equation*}%
The norm in $B_{\rho _{_{\gamma }}}$ is defined as%
\begin{equation*}
\left\Vert f\right\Vert _{\rho _{_{\gamma }}}=\sup_{x\in 
%TCIMACRO{\U{211d} }%
%BeginExpansion
\mathbb{R}
%EndExpansion
_{+}}\frac{\left\vert f(x)\right\vert }{\rho _{_{\gamma }}(x)}.
\end{equation*}

We can give some estimations of the errors $\left\vert L_{n}^{\ast (\alpha
,\beta )}(f;q,x)-f(x)\right\vert $, $n\in N$, for unbounded functions by
using a weighted modulus of smoothness associated to the space $B_{\rho
_{\gamma }}(%
%TCIMACRO{\U{211d} }%
%BeginExpansion
\mathbb{R}
%EndExpansion
_{+}).$

We consider%
\begin{equation}
\Omega _{\rho _{_{\gamma }}}(f;\delta )=\underset{x\geq 0,\ 0<h\leq \delta }{%
\sup }\frac{\left\vert f(x+h)-f(x)\right\vert }{1+(x+h)^{2+\gamma }}~,\ \ \
\ \delta >0,\ \ \ \gamma \geq 0.  \tag{13}
\end{equation}%
It is evident that for each $f\in B_{\rho _{\gamma }}(%
%TCIMACRO{\U{211d} }%
%BeginExpansion
\mathbb{R}
%EndExpansion
_{+})\,,$ $\Omega _{\rho _{\gamma }}(f;.)$ is well defined and\ \ 
\begin{equation*}
\Omega _{\rho _{\gamma }}(f;\delta )\leq 2\left\Vert f\right\Vert _{\rho
_{\gamma }}.
\end{equation*}%
The weighted modulus of smoothness $\Omega _{\rho _{\gamma }}(f;.)$ posseses
the following properties.

\begin{equation}
\Omega _{\rho _{\gamma }}(f;\lambda \delta )\leq (\lambda +1)\Omega _{\rho
_{\gamma }}(f;\delta ),\ \ \delta >0,\ \lambda >0,  \tag{14}
\end{equation}

\begin{equation*}
\Omega _{\rho _{\gamma }}(f;n\delta )\leq n\Omega _{\rho _{\gamma
}}(f;\delta ),\ n\in N,
\end{equation*}

\begin{equation*}
\underset{\delta \rightarrow 0^{+}}{\lim }\Omega _{\rho _{\gamma }}(f;\delta
)=0.
\end{equation*}

As it is known, weighted Korovkin type theorems have been proven by
Gadjiev(see [9]).

\begin{theorem}
Let $q\in (0,1)$ and $\gamma \geq 0$. For all non-decreasing $f\in B_{\rho
_{\gamma }}(%
%TCIMACRO{\U{211d} }%
%BeginExpansion
\mathbb{R}
%EndExpansion
_{+})$ we have%
\begin{equation*}
\left\vert L_{n}^{\ast (\alpha ,\beta )}(f;q,x)-f(x)\right\vert \leq \sqrt{%
L_{n}^{\ast (\alpha ,\beta )}(\mu _{x,\gamma }^{2};x)}\left( 1+\frac{1}{%
\delta }\sqrt{L_{n}^{\ast (\alpha ,\beta )}(\psi _{x}^{2};x)}\right) \Omega
_{\rho _{\gamma }}(f;\delta ),
\end{equation*}%
$x\geq 0,~\delta >0,~n\in N$, where $\mu _{x,\gamma }(t):=1+\left(
x+\left\vert t-x\right\vert \right) ^{2+\gamma },\ \ \psi
_{x}(t):=\left\vert t-x\right\vert ,\ \ t\geq 0$.
\end{theorem}

\begin{proof}
Let $n\in N$ and $f\in B_{\rho _{\gamma }}(%
%TCIMACRO{\U{211d} }%
%BeginExpansion
\mathbb{R}
%EndExpansion
_{+}).$ From (13) and (14), we can write%
\begin{eqnarray*}
\left\vert f(t)-f(x)\right\vert &\leq &\left( 1+\left( x+\left\vert
t-x\right\vert \right) ^{2+\gamma }\right) \left( 1+\frac{1}{\delta }%
\left\vert t-x\right\vert \right) \Omega _{\rho _{\gamma }}(f;\delta ) \\
&=&\mu _{x,\gamma }(t)\left( 1+\frac{1}{\delta }\psi _{x}(t)\right) \Omega
_{\rho _{\gamma }}(f;\delta ).
\end{eqnarray*}%
Taking into account the definition of $q-$ integration, we get%
\begin{equation}
\dint\limits_{q\left( \frac{\left[ k\right] _{q}+q^{k-1}\alpha }{\left[ n%
\right] _{q}+\beta }\right) }^{\frac{\left[ k+1\right] _{q}+q^{k}\alpha }{%
\left[ n\right] _{q}+\beta }}f\left( q^{-k+1}t\right)
d_{q}t=q^{k-1}\dint\limits_{\frac{\left[ k\right] _{q}+q^{k-1}\alpha }{%
q^{k-2}\left( \left[ n\right] _{q}+\beta \right) }}^{\frac{\left[ k+1\right]
_{q}+q^{k}\alpha }{q^{k-1}\left( \left[ n\right] _{q}+\beta \right) }%
}f\left( t\right) d_{q}t.  \tag{15}
\end{equation}%
Consequently, the operators $L_{n}^{\ast (\alpha ,\beta )}$ can be expressed
as follows%
\begin{equation*}
L_{n}^{\ast (\alpha ,\beta )}(f;q,x)=\left( \left[ n\right] _{q}+\beta
\right) \dsum\limits_{k=0}^{\infty }q^{\frac{k(k-1)}{2}}\frac{%
D_{q}^{k}\left( \varphi _{n}(x)\right) }{\left[ k\right] _{q}!}\left(
-x\right) ^{k}q^{k-1}\dint\limits_{\frac{\left[ k\right] _{q}+q^{k-1}\alpha 
}{q^{k-2}\left( \left[ n\right] _{q}+\beta \right) }}^{\frac{\left[ k+1%
\right] _{q}+q^{k}\alpha }{q^{k-1}\left( \left[ n\right] _{q}+\beta \right) }%
}f\left( t\right) d_{q}t.
\end{equation*}%
By using the Cauchy-Schwartz inequality and (15), we obtain%
\begin{eqnarray*}
&&\left\vert L_{n}^{\ast (\alpha ,\beta )}(f;q,x)-f(x)\right\vert \\
&\leq &\left( \left[ n\right] _{q}+\beta \right) \dsum\limits_{k=0}^{\infty
}q^{\frac{k(k-1)}{2}}\frac{D_{q}^{k}\left( \varphi _{n}(x)\right) }{\left[ k%
\right] _{q}!}\left( -x\right) ^{k}q^{k-1}\dint\limits_{\frac{\left[ k\right]
_{q}+q^{k-1}\alpha }{q^{k-2}\left( \left[ n\right] _{q}+\beta \right) }}^{%
\frac{\left[ k+1\right] _{q}+q^{k}\alpha }{q^{k-1}\left( \left[ n\right]
_{q}+\beta \right) }}\left\vert f\left( t\right) -f(x)\right\vert d_{q}t \\
&\leq &\left( L_{n}^{\ast (\alpha ,\beta )}(\mu _{x,\gamma };x)+\frac{1}{%
\delta }L_{n}^{\ast (\alpha ,\beta )}(\mu _{x,\gamma }\psi _{x};x)\right)
\Omega _{\rho _{\gamma }}(f;\delta ) \\
&\leq &\sqrt{L_{n}^{\ast (\alpha ,\beta )}(\mu _{x,\gamma }^{2};x)}\left( 1+%
\frac{1}{\delta }\sqrt{L_{n}^{\ast (\alpha ,\beta )}(\psi _{x}^{2};x)}%
\right) \Omega _{\rho _{\gamma }}(f;\delta ).
\end{eqnarray*}
\end{proof}

\begin{lemma}
For $m\in N$ and $q\in (0,1)$ we have%
\begin{equation*}
L_{n}^{\ast (\alpha ,\beta )}(e_{m};q,x)\leq A_{m,q}\left( 1+x^{m}\right) \
,\ \ x\in 
%TCIMACRO{\U{211d} }%
%BeginExpansion
\mathbb{R}
%EndExpansion
_{+},\ \ n\in N,
\end{equation*}%
where $A_{m,q}$ is a positive constant depending only on $m,~\alpha $ and $q$%
.
\end{lemma}

\begin{proof}
For $k\in N$ and $0<q<1$ the following inequality holds true%
\begin{equation}
1\leq \left[ k+1\right] _{q}\leq 2\left[ k\right] _{q}.  \tag{16}
\end{equation}%
Thus, for $m\in N,$ from (1) and (16) we get%
\begin{eqnarray*}
L_{n}^{\alpha ,\beta }(e_{m};q,x) &=&\dsum\limits_{k=0}^{\infty }q^{\frac{%
k(k-1)}{2}}\frac{D_{q}^{k}\left( \varphi _{n}(x)\right) }{\left[ k\right]
_{q}!}\left( -x\right) ^{k}\frac{1}{q^{km-m}}\left( \frac{\left[ k\right]
_{q}+q^{k-1}\alpha }{\left[ n\right] _{q}+\beta }\right) ^{m} \\
&=&\frac{x\left[ n\right] _{q}}{\left[ n\right] _{q}+\beta }%
\dsum\limits_{k=0}^{\infty }q^{\frac{k(k-1)}{2}}\frac{D_{q}^{k}\left(
\varphi _{m(n)}(x)\right) \left( 1+\alpha _{k,n}(x)\right) }{\left[ k\right]
_{q}!}\left( -x\right) ^{k}\frac{1}{q^{k(m-1)}}\left( \frac{\left[ k\right]
_{q}+q^{k}\alpha }{\left[ n\right] _{q}+\beta }\right) ^{m-1} \\
&&+\frac{\alpha }{\left[ n\right] _{q}+\beta }\dsum\limits_{k=0}^{\infty }q^{%
\frac{k(k-1)}{2}}\frac{D_{q}^{k}\left( \varphi _{n}(x)\right) }{\left[ k%
\right] _{q}!}\left( -x\right) ^{k}\left( \frac{\left[ k\right]
_{q}+q^{k-1}\alpha }{q^{k-1}\left( \left[ n\right] _{q}+\beta \right) }%
\right) ^{m-1} \\
&\leq &\frac{x\left[ n\right] _{q}}{\left[ n\right] _{q}+\beta }\varphi
_{m(n)}(x)\left( 1+\alpha _{0,n}(x)\right) \left( \frac{1+\alpha }{\left[ n%
\right] _{q}+\beta }\right) ^{m-1} \\
&&+\frac{x\left[ n\right] _{q}}{\left[ n\right] _{q}+\beta }%
\dsum\limits_{k=0}^{\infty }q^{\frac{k(k-1)}{2}}\frac{D_{q}^{k}\left(
\varphi _{m(n)}(x)\right) \left( 1+\alpha _{k,n}(x)\right) }{\left[ k\right]
_{q}!}\left( -x\right) ^{k}\left( \frac{2\left[ k\right] _{q}+q^{k}\alpha }{%
q^{k}\left( \left[ n\right] _{q}+\beta \right) }\right) ^{m-1} \\
&&+\frac{\alpha }{\left[ n\right] _{q}+\beta }L_{n}^{\alpha ,\beta
}(e_{m-1};q,x) \\
&=&\frac{x\left[ n\right] _{q}}{\left[ n\right] _{q}+\beta }\varphi
_{m(n)}(x)\left( 1+\alpha _{0,n}(x)\right) \left( \frac{1+\alpha }{\left[ n%
\right] _{q}+\beta }\right) ^{m-1} \\
&&+\frac{x\left[ n\right] _{q}}{\left[ n\right] _{q}+\beta }\left( \frac{2}{q%
}\right) ^{m-1}L_{n+1}^{\alpha ,\beta }(e_{m-1};q,x)+\frac{\alpha }{\left[ n%
\right] _{q}+\beta }L_{n}^{\alpha ,\beta }(e_{m-1};q,x) \\
&\leq &x+\left( \frac{2}{q}\right) ^{m-1}\frac{1}{\left[ n\right] _{q}+\beta 
}\left( x\left[ n\right] _{q}+\alpha q^{m-1}\right) L_{n+1}^{\alpha ,\beta
}(e_{m-1};q,x) \\
&\leq &2m\left( \frac{2}{q}\right) ^{m-1}\frac{1}{\left[ n\right] _{q}+\beta 
}\left( \left[ n\right] _{q}(1+x^{m})+\alpha ^{m}q^{\frac{m(m-1)}{2}}\right)
.
\end{eqnarray*}%
based on the above inequality and by using the mathematical induction over $%
m\in N$, we obtain%
\begin{equation*}
L_{n}^{\alpha ,\beta }(e_{m};q,x)\leq B_{m,q}\left( 1+x^{m}\right) ,
\end{equation*}%
$x\in 
%TCIMACRO{\U{211d} }%
%BeginExpansion
\mathbb{R}
%EndExpansion
_{+},\ n\in N,$\ where%
\begin{equation}
B_{m,q}:=2m\left( \frac{2}{q}\right) ^{\frac{m\left( m-1\right) }{2}}\left(
1+\alpha ^{m}q^{\frac{m(m-1)}{2}}\right) .  \tag{17}
\end{equation}%
On the other hand,%
\begin{eqnarray*}
L_{n}^{\ast (\alpha ,\beta )}(e_{m};q,x) &=&\left( \left[ n\right]
_{q}+\beta \right) \dsum\limits_{k=0}^{\infty }q^{\frac{k(k-1)}{2}}\frac{%
D_{q}^{k}\left( \varphi _{n}(x)\right) }{\left[ k\right] _{q}!}\left(
-x\right) ^{k}\dint\limits_{q\left( \frac{\left[ k\right] _{q}+q^{k-1}\alpha 
}{\left[ n\right] _{q}+\beta }\right) }^{\frac{\left[ k+1\right]
_{q}+q^{k}\alpha }{\left[ n\right] _{q}+\beta }}e_{m}\left( q^{-k+1}t\right)
d_{q}t \\
&=&\frac{\left[ n\right] _{q}+\beta }{\left( \left[ n\right] _{q}+\beta
\right) ^{m+1}}\dsum\limits_{k=0}^{\infty }q^{\frac{k(k-1)}{2}}\frac{%
D_{q}^{k}\left( \varphi _{n}(x)\right) }{\left[ k\right] _{q}!}\left(
-x\right) ^{k}\frac{q^{-km+m}}{\left[ m+1\right] _{q}} \\
&&\times \left\{ \left( \left[ k+1\right] _{q}+q^{k}\alpha \right)
^{m+1}-q^{m+1}\left( \left[ k\right] _{q}+q^{k-1}\alpha \right)
^{m+1}\right\} .
\end{eqnarray*}%
Since%
\begin{eqnarray*}
&&\left( \left[ k+1\right] _{q}+q^{k}\alpha \right) ^{m+1}-q^{m+1}\left( %
\left[ k\right] _{q}+q^{k-1}\alpha \right) ^{m+1} \\
&=&\left( \left( \left[ k+1\right] _{q}+q^{k}\alpha \right) ^{m}+q\left( %
\left[ k+1\right] _{q}+q^{k}\alpha \right) ^{m-1}\left( \left[ k\right]
_{q}+q^{k-1}\alpha \right) +...+q^{m}\left( \left[ k\right]
_{q}+q^{k-1}\alpha \right) ^{m}\right) \\
&\leq &\left( m+1\right) \left( \left[ k+1\right] _{q}+q^{k}\alpha \right)
^{m} \\
&\leq &\left( m+1\right) 2^{m}\left( \left[ k\right] _{q}+q^{k}\alpha
\right) ^{m},\ k\in N,
\end{eqnarray*}%
from condition $(vi)$, we can write%
\begin{eqnarray*}
L_{n}^{\ast (\alpha ,\beta )}(e_{m};q,x) &\leq &\frac{\varphi
_{n}(x)(m+1)(1+\alpha )^{m}q^{m}}{\left( \left[ n\right] _{q}+\beta \right)
^{m}\left[ m+1\right] _{q}}+\frac{2^{m}(m+1)}{\left[ m+1\right] _{q}}%
L_{n}^{\alpha ,\beta }(e_{m};q,x) \\
&\leq &A_{m,q}\left( 1+x^{m}\right) ,
\end{eqnarray*}%
where $A_{m,q}:=\frac{(m+1)(1+\alpha )^{m}q^{m}}{\left[ m+1\right] _{q}}+%
\frac{2^{m}(m+1)}{\left[ m+1\right] _{q}}B_{m,q}$ and $B_{m,q}$ is given by
(17).
\end{proof}

\begin{remark}
Since any linear \ positive operator is monotone, from Lemma 9\ we can
easily see that $L_{n}^{\ast (\alpha ,\beta )}(f;q,.)\in B_{\rho _{\gamma }}(%
%TCIMACRO{\U{211d} }%
%BeginExpansion
\mathbb{R}
%EndExpansion
_{+})$ for each $f\in B_{\rho _{\gamma }}(%
%TCIMACRO{\U{211d} }%
%BeginExpansion
\mathbb{R}
%EndExpansion
_{+}),\ \gamma \in N_{0}.$
\end{remark}

\begin{theorem}
Let $f\in B_{\rho _{\gamma }}(%
%TCIMACRO{\U{211d} }%
%BeginExpansion
\mathbb{R}
%EndExpansion
_{+})$ be a non-decreasing function$,$ then%
\begin{equation*}
\left\Vert L_{n}^{\ast (\alpha ,\beta )}(f;q_{n},.)-f\right\Vert _{\rho
_{\gamma +1}}\leq K_{\gamma ,q_{0}}\Omega _{\rho _{\gamma }}(f;\delta _{n}),
\end{equation*}%
where $\delta _{n}:=\sqrt{\frac{\left[ n\right] _{q_{n}}\eta _{n}+1}{%
q_{n}\left( \left[ n\right] _{q_{n}}+\beta \right) }}$ and $K_{\gamma
,q_{0}} $ is a positive constant independent on $f$ and $n$.
\end{theorem}

\begin{proof}
The identities (3)-(5) imply%
\begin{eqnarray*}
L_{n}^{\ast (\alpha ,\beta )}(\psi _{x}^{2};q_{n},x) &=&L_{n}^{\ast (\alpha
,\beta )}(\left( t-x\right) ^{2};q_{n},x) \\
&=&\frac{\left[ n\right] _{q_{n}}\left[ m(n)\right] _{q_{n}}}{q_{n}\left( %
\left[ n\right] _{q_{n}}+\beta \right) ^{2}}\left( 1+\alpha
_{1,m(n)}(x)\right) \left( 1+\alpha _{2,n}(x)\right) x^{2} \\
&&+\frac{\left[ n\right] _{q_{n}}\left[ \left[ 3\right] _{q_{n}}+q_{n}\left(
(1+3\alpha )\left[ 2\right] _{q_{n}}+1\right) \right] }{\left[ 3\right]
_{q_{n}}\left( \left[ n\right] _{q_{n}}+\beta \right) ^{2}}\left( 1+\alpha
_{1,n}(x)\right) x \\
&&+\frac{q_{n}^{2}(1+3\alpha +3\alpha ^{2})}{\left[ 3\right] _{q_{n}}\left( %
\left[ n\right] _{q_{n}}+\beta \right) ^{2}}-2x\left\{ \frac{\left[ n\right]
_{q_{n}}}{\left[ n\right] _{q_{n}}+\beta }x\left( 1+\alpha _{1,n}(x)\right) +%
\frac{q_{n}(1+2\alpha )}{\left[ 2\right] _{q_{n}}\left( \left[ n\right]
_{q_{n}}+\beta \right) }\right\} +x^{2} \\
&\leq &\frac{\left( \left[ n\right] _{q_{n}}\eta _{n}(x)+1+\beta \right)
x^{2}}{q_{n}\left( \left[ n\right] _{q_{n}}+\beta \right) }+\frac{2(3\alpha
+3)}{q_{n}\left( \left[ n\right] _{q_{n}}+\beta \right) }x+\frac{1+3\alpha
+3\alpha ^{2}}{q_{n}\left( \left[ n\right] _{q_{n}}+\beta \right) } \\
&\leq &\frac{9(1+\beta )^{2}\rho _{0}(x)}{q_{n}\left( \left[ n\right]
_{q_{n}}+\beta \right) }\left\{ \left[ n\right] _{q_{n}}\eta
_{n}(x)+1\right\} 
\end{eqnarray*}%
where $\eta _{n}(x):=\max \left\{ \alpha _{1,n}(x),\alpha
_{1,m(n)}(x),\alpha _{2,n}(x)\right\} $.\newline
Since $\eta _{n}(x)$ converges uniformly to zero, we have $\eta _{n}=\sup
\eta _{n}(x)$ such that $\eta _{n}$ converges to zero as $n\rightarrow
\infty .$ Let $\gamma \in N_{0}$ and $f\in B_{\rho _{\gamma }}(%
%TCIMACRO{\U{211d} }%
%BeginExpansion
\mathbb{R}
%EndExpansion
_{+})$ be a fixed function. From Theorem 8 and above inequality, we can write%
\begin{eqnarray*}
&&\frac{\left\vert L_{n}^{\ast (\alpha ,\beta )}(f;q_{n},x)-f(x)\right\vert 
}{\rho _{\gamma +1}(x)} \\
&\leq &\sqrt{\frac{L_{n}^{\ast (\alpha ,\beta )}(\mu _{x,\gamma }^{2};q,x)}{%
\rho _{\gamma +1}^{2}(x)}}\left( 1+\frac{1}{\delta _{n}}\sqrt{L_{n}^{\ast
(\alpha ,\beta )}(\psi _{x}^{2};q_{n},x)}\right) \Omega _{\rho _{\gamma
}}(f;\delta _{n}) \\
&\leq &\sqrt{\frac{L_{n}^{\ast (\alpha ,\beta )}(\mu _{x,\gamma
}^{2};q_{n},x)\rho _{0}(x)}{\rho _{\gamma +1}^{2}(x)}}\left( 1+\frac{1}{%
\delta _{n}}\sqrt{\frac{9(1+\beta )^{2}\rho _{0}(x)}{q_{n}\left( \left[ n%
\right] _{q_{n}}+\beta \right) }\left\{ \left[ n\right] _{q_{n}}\eta
_{n}(x)+1\right\} }\right) \Omega _{\rho _{\gamma }}(f;\delta _{n}) \\
&\leq &12(1+\beta )\sqrt{\frac{L_{n}^{\ast (\alpha ,\beta )}(\mu _{x,\gamma
}^{2};q_{n},x)}{\rho _{2(\gamma +1)}(x)}}\left( 1+\frac{1}{\delta _{n}}\sqrt{%
\frac{\left[ n\right] _{q_{n}}\eta _{n}(x)+1}{q_{n}\left( \left[ n\right]
_{q_{n}}+\beta \right) }}\right) \Omega _{\rho _{\gamma }}(f;\delta _{n}).
\end{eqnarray*}%
Since%
\begin{eqnarray*}
\ \mu _{x,\gamma }^{2}(t) &=&\left( 1+\left( x+\left\vert t-x\right\vert
\right) ^{2+\gamma }\right) ^{2}\leq 2\left( 1+\left( 2x+t\right)
^{4+2\gamma }\right)  \\
&\leq &2\left( 1+2^{4+2\gamma }\left( (2x)^{4+2\gamma }+t^{4+2\gamma
}\right) \right) ,
\end{eqnarray*}%
from Lemma 9, we get%
\begin{equation*}
L_{n}^{\ast (\alpha ,\beta )}(\mu _{x,\gamma }^{2};q,x)\leq \lambda _{\gamma
,q_{n}}^{2}\rho _{2(\gamma +1)}(x),
\end{equation*}%
where $\lambda _{\gamma ,q_{n}}^{2}=2^{5+2\gamma }\left( 2^{4+2\gamma
}+A_{4+2\gamma ,q_{n}}\right) .$ Choosing $\delta _{n}:=\sqrt{\frac{\left[ n%
\right] _{q_{n}}\eta _{n}+1}{q_{n}\left( \left[ n\right] _{q_{n}}+\beta
\right) }}$ and $K_{\gamma ,q_{0}}:=24(1+\beta )\lambda _{\gamma ,q_{0}}$,
where $q_{0}:=\underset{n\in N}{\min }q_{n}$, the proof is finished.
\end{proof}


\begin{thebibliography}{99}
\bibitem{} U. Abel, V. Gupta, An estimate of the rate of convergence of a
Bezier variant of the Baskakov-Kantorovich operators for bounded variation
functions, Demonstratio math.,36, 123-136, 2003.

\bibitem{} A. Aral, V. Gupta, On the Durrmeyer type modification of the $q-$%
Baskakov type operators, Nonlinear Analysis 72, 1171-1180, 2010.

\bibitem{} \c{C}. Atakut, On Kantorovich-Baskakov-Stancu type operators,
Bull. Cal. Math. Soc., 91(2), 149-156, 1999.

\bibitem{} V.A. Baskakov, An example of a sequence of linear positive
operators in the space of continuous, DAN, 113, 249-251, 1957.

\bibitem{} \.{I}.B\"{u}y\"{u}kyaz\i c\i , \c{C}. Atakut, On Stancu type
generalization of $q-$Baskakov operators,52(5-6), 752-759, 2010.

\bibitem{} \.{I}. B\"{u}y\"{u}kyaz\i c\i , Direct and inverse results for
generalized $q-$Bernstein polynomials, Far East J. Appl. Math., 34(2), 191
-- 204, 2009.

\bibitem{} R.A. DeVore, G.G. Lorentz, Constructive Approximation, Springer,
Berlin, 1993.

\bibitem{} A.D. Gadjiev, \c{C}. Atakut, On approximation of unbounded
functions by the generalized Baskakov operators. Trans. Acad. Sci. Azerb.
Ser. Phys.-Tech. Math. Sci., 23(1), Math. Mech., 33--42, 2003.

\bibitem{} A.D. Gadjiev, Theorems of the type of P. P. Korovkin's theorems,
Math. Zametki, 20(5), 781-786, 1976, .

\bibitem{} N. K. Govil, \ V. Gupta, Convergence of $q-$Meyer-K\"{o}%
nig-Zeller-Durrmeyer operators. Adv. Stud. Contemp. Math. (Kyungshang),
19(1), 97--108, 2009.

\bibitem{} V. Gupta, Z. Finta, On certain $q-$Durrmeyer type operators,
Appl. Math. Comput., 209(2), 415-420, 2009.

\bibitem{} V. Gupta, C. Radu, Statistical approximation properties of $q-$%
Baskakov-Kantorovich operators, Cent. Eur. J. Math., 7(4), 809-818, 2009.

\bibitem{} S.C. Jing, H.Y. Fan, $q-$ Taylor's formula with its $q-$
remainder, Comm. Theoret. Phys., 23(1), 117-120, 1995.

\bibitem{} V.G. Kac, P. Cheung, Quantum Calculus, Universitext,
Springer-Verlag, New York, 2002.

\bibitem{} G. M. Phillips, Bernstein polynomials based on the $q-$integers,
Ann. Numer. Math., 4, 511-518, 1997.

\bibitem{} G.\ M.\ Phillips, A survey of results on the $q-$ Bernstein
polynomials, IMA Journal of Numerical Analysis, 30, 277--288, 2010.

\bibitem{} G.M. Phillips, Interpolation and Approximation by Polynomials,
Springer-Verlag, New York, 2003.

\bibitem{} C. Radu, Statistical approximation properties of Kantorovich
operators based on $q-$integers, Creat. Math. Inform., 17, 75-84, 2008.
\end{thebibliography}
\end{document}